\documentclass[11pt,letterpaper]{amsart}

\usepackage{amsmath}
\usepackage{amssymb}
\usepackage{amsfonts}
\usepackage{enumerate}
\usepackage{amsthm}
\usepackage{esint}
\usepackage{bbm}

\newtheorem{theorem}{Theorem}[section]
\newtheorem{lemma}[theorem]{Lemma}

\theoremstyle{definition}
\newtheorem{definition}[theorem]{Definition}
\newtheorem{remark}[theorem]{Remark}

\newtheorem{assumption}[theorem]{Assumption}

\numberwithin{equation}{section}

\newcommand{\rmi}{\mathrm{i}}
\newcommand{\rme}{\mathrm{e}}

\newcommand{\R}{\mathbb{R}}
\newcommand{\N}{\mathbb{N}}

\newcommand{\C}{\mathbb{C}}

\usepackage{mathtools}

\title[Unique determination of polyhedra]{On unique determination of polyhedral sets}

\author[L.~Rondi]{Luca Rondi}

\address[L.~Rondi]{Dipartimento di Matematica, Universit\`a degli Studi di Pavia, Italy}
\email{luca.rondi@unipv.it}

\begin{document}

\setcounter{section}{0}
\setcounter{secnumdepth}{2}

\begin{abstract}
In this paper we develop in detail the geometric constructions that lead to many uniqueness results for the determination of polyhedral sets, typically scatterers, by a finite minimal number of measurements. We highlight how unique continuation and a suitable reflection principle are enough to proceed with the constructions, without any other assumption on the underlying partial differential equation or the boundary condition. We also aim to keep the geometric constructions and their proofs as simple as possible.

To illustrate the applicability of this theory,
we show how several uniqueness results present in the literature immediately follow from our arguments. Indeed we believe that this theory may 
serve as a roadmap for establishing similar uniqueness results for other partial differential equations or boundary conditions.

\medskip


\medskip

\end{abstract}

\maketitle

\section{Introduction}
One of the most classical and important open questions in the inverse scattering problems field is the following: How many (far-field) measurements are needed to uniquely
determine a suitable scatterer? In other words, how many experiments, corresponding to different incident waves, are needed to uniquely determine a suitable scatterer?
Here we limit ourselves to the time-harmonic case and to incident waves of planar type. As a matter of notation, let us point out that a \emph{scatterer} is just a compact set whose complement is connected. A scatterer is called an \emph{obstacle} if it the closure of its interior and \emph{screen} if its interior is empty. A nice introduction to inverse scattering problems, in the acoustic and electromagnetic cases, may be found for instance in \cite{CK}.

Even in the simplest case of acoustic waves with sound-soft obstacles, the answer is not yet complete. By Schiffer's theorem, we know that infinitely many incident waves (with the same frequency) are enough. With a given bound on the diameter $R$ of the region containing the unknown obstacle, finitely many incident waves are enough, a fact first noted in \cite{Col-Sle}; in this case the number of measurements just depends on $R$ and the frequency of the incident waves and it is $1$ provided such a frequency is small enough with respect to $R$. If the scatterer includes screens and obstacles, this result has been extended in \cite{Ron03}, where some further although minimal regularity assumptions on the unknown scatterer are required. However, there is a long standing conjecture that one measurement should suffice, independently of the frequency of the incident wave
or the size of the scatterer, at least in the obstacle case.

In order to approach such a conjecture, and to possibly find the minimal number of measurements needed in other inverse scattering problems, unique determination of scatterers belonging to some special classes has been studied. In this respect, one of the most successful special classes has been so far the one of so-called polyhedral scatterers, that is, scatterers whose boundary is the finite union of cells, a cell being the closure of an open subset of a hyperplane. We note that a polyhedral obstacle is just the finite union of polyhedra.

The first breakthrough in this direction was obtained in \cite{C-Y}, even if some earlier results may be found in the unpublished work \cite{Liu-Nach}. In \cite{C-Y} it was proved uniqueness for polygonal obstacles satisfying a non-trapping condition, with one measurement in the sound-soft case and two measurements in the sound-hard case. For sound-soft scatterers, this result has been extended in \cite{Ale-Ron} on several aspects. In fact, it was proved that one measurement guarantees uniqueness
in any dimension and for general sound-soft polyhedral scatterers, thus removing both the obstacle and the non-trapping assumptions.

After \cite{Ale-Ron},
an enormous amount of research has been done in the field, pursuing extensions to other boundary conditions
or to other kinds of waves, like electromagnetic waves or elastic waves. We just mention here the earliest most significant results.

For general polyhedral scatterers, uniqueness was proved with $N$ measurements, $N$ being the dimension of the space, first in \cite{Liu-Zou} in the sound-hard case and then in \cite{Liu-Zou2} in the mixed sound-soft and sound-hard case. Uniqueness with one measurement for polyhedral obstacles was proved in the sound-hard
case in \cite{Els-Yam1} for $N=2$ and in \cite{Els-Yam2} for any $N\geq 2$. The two-dimensional result has been extended in \cite{Liu-Zou2} to the mixed sound-soft and sound-hard case.

The electromagnetic case has been treated in \cite{Liu3} and \cite{Liu}, whereas the extension to elastic waves has been developed in \cite{Els-Yam2}.

In a different and significant line of research, stemmed from the breakthrough result of \cite{corners},
polyhedral structures, in particular corners and edges, play an important role also in the determination of the support of penetrable
obstacles by a finite number of measurements.

Despite the different equations or boundary conditions, it is clear that all these uniqueness results for the determination of (non-penetrable) polyhedral scatterers
share two crucial common features. Namely, unique continuation properties for the solutions to the equation and a suitable reflection principle depending on the boundary condition. These are combined in the development of suitable geometric constructions. In this paper we highlight that indeed these two conditions, summarised in Assumption~\ref{assum1}, are enough to develop these geometric constructions. Moreover, once Assumption~\ref{assum1} is satisfied, the geometric constructions are completely independent of the other properties of the equation or of the boundary condition. Moreover, our effort is to keep these constructions as simple and as general as possible in such a way to
make completely transparent the geometrical procedure involved
in these uniqueness results. In Section~\ref{sec2}, we state and prove our two main results, Theorem~\ref{thm1} for general polyhedral scatterers and Theorem~\ref{thm2} for polyhedral obstacles. We also aim to keep proofs as simple as possible and we fully succeeded at least for Theorem~\ref{thm1}. Other important features of our results are the following.
\begin{itemize}
\item We can treat at the same time any (finite) number of measurements, simply by changing $M$ in the definition of the differential operator $A$.
\item We can treat at the same time different boundary conditions, that is, we can treat mixed boundary conditions, see Section~\ref{sec2bis}, in particular Theorem~\ref{thm3}.
\end{itemize}

In order to show the power of our general theory, we obtain in a single shot all previously described uniqueness results, possibly with some extensions in the mixed boundary conditions cases. These applications are developed in Section~\ref{sec3}, first for the acoustic and then for the electromagnetic and elastic waves. We thus believe that our general constructions will be useful in order to obtain similar uniqueness results for other kinds of equations or boundary conditions.

Finally, we wish to make some comments on the corresponding stability results. The first stability result has been obtained in \cite{Ron08} for sound-soft general polyhedral scatterers with one measurement. Stability has been extended to sound-hard general polyhedral scatterers with $N$ measurements and to sound-hard polyhedral obstacles with $1$ measurement in \cite{rondi:scatt}. Finally, stability for the electromagnetic case has been proved in \cite{LRX} with $2$ measurements for general polyhedral scatterers and $1$ measurement for polyhedral obstacles. In all these stability results, quantitative versions of Theorem~\ref{thm1} for general scatterers and of Theorem~\ref{thm2} for obstacles are needed. The proof of Theorem~\ref{thm2} is much more difficult than the one of Theorem~\ref{thm1}, likewise its quantitative version is extremely involved, we refer the interested reader to \cite{LRX} for details. Here we limit ourselves to consider quantitative versions of Theorem~\ref{thm1} and to stress some of their common features. In fact, we believe that a general stability result may be obtained independently of the equation and the boundary condition provided the following kinds of properties are satisfied.
\begin{enumerate}[(1)]
\item The differential operator $A$ satisfies interior and boundary regularity estimates.
\item The differential operator $A$ satisfies quantitative unique continuation results, like a three-spheres inequality with optimal exponents.
\item There are uniforms bounds on the solutions to the scattering problem which are independent on the scatterer for scatterers belonging to suitable admissible classes.
\item The boundary condition satisfies a quantitative version of the reflection principle, which often reduces to quantitative estimates on solutions with respect to Cauchy data.
\end{enumerate}
As one can easily check, these are the common essential ingredients of all the stability results for general polyhedral scatterers, that is, are the essential ingredients to obtain a quantitative analogue of Theorem~\ref{thm1}. We finally note that point (3) above might be challenging for the following reasons. First, dealing with polyhedral structures we can not expect to have better than Lipschitz regularity on the scatterers. Moreover, when allowing obstacles and screens or combinations of both, regularity might be even weaker. For the Dirichlet boundary condition in the acoustic case, uniform bounds can be proved in a relatively easy way even for quite general scatterers, see for instance \cite{Ron03}. For the Neumann boundary condition in the acoustic case or for the electromagnetic case, uniform bounds for general scatterers are much harder to obtain. In this respect, the strategy developed in \cite{rondi:ac}, and optimised in \cite{F-R}, for Neumann in the acoustic case might be a good starting point. In fact, the same strategy proved successful in treating the electromagnetic case as well, see \cite{LRX}.

\medskip

\noindent
\textbf{Acknowledgement}\\
The author acknowledges support by GNAMPA, INdAM, through 2023 projects.

\section{The main strategies for uniqueness results}\label{sec2}

The integer $N\geq 2$ denotes the dimension of the space. For any $x\in\R^N$ and any $r>0$, $B_r(x)$ is the open ball with center $x$ and radius $r$. We use also the notation $B_r=B_r(0)$.

For any $(N-1)$-dimensional hyperplane $\Pi$, we call $T_{\Pi}$ the reflection in $\Pi$. Moreover, fixed a unit vector $\nu$ orthogonal to $\Pi$, we call $H^+_{\Pi}$ and 
$H^-_{\Pi}$ the two connected components, actually half-spaces, of $\R^N\backslash \Pi$, with $\nu$ being the exterior unit normal on $\Pi$ for $H^+_{\Pi}$ and the interior unit normal on $\Pi$ for $H^-_{\Pi}$.

We consider the following general framework.
Let us fix positive integers $M$, $L$, $L_1$, $J$ and $J_1$. Let $D\subset \R^N$ be an open set and let $u=(u^1,\ldots,u^M):D\subset\R^N\to \C^M$ be a complex vector-valued function.

Let $A_k$, $k=1,\ldots,L$, be the following differential operators with constant coefficients
$$A_ku=\sum_{j=1}^M\sum_{|\alpha|\leq J}a^j_{\alpha,k}D^{\alpha}u^{j}$$
and let us define
$$Au=(A_1u,\ldots ,A_Lu).$$
As usual, $\alpha$ denotes an $N$-multiindex and $|\alpha|$ is its weight. The constants $a^j_{\alpha,k}$ are complex numbers.

Let $B_k$, $k=1,\ldots,L_1$, be the following boundary operators with constant coefficients
$$B_ku=\sum_{i=1}^N\sum_{j=1}^M\sum_{|\alpha|\leq J_1}\nu_i b^j_{\alpha,i,k}D^{\alpha}u^{j}+\sum_{j=1}^M\sum_{|\alpha|\leq J_1}b^j_{\alpha,0,k}D^{\alpha}u^{j}$$
and let us define
$$Bu=(B_1u,\ldots ,B_{L_1}u).$$
As usual, $\nu$ denotes the unit exterior normal vector on the boundary of $D$, where it is well-defined. The constants $b^j_{\alpha,i,k}$ are complex numbers.

On this general structure we impose the following.
\begin{assumption}\label{assum1}
We assume that the operator $A$ and the boundary condition $B$ satisfy the following.
\begin{enumerate}[(\text{A}.1)]
\item If $u$ satisfies $Au=0$ in an open set $D$, then all partial derivatives $D^{\alpha}u^j$ appearing in the definitions of $A$ and $B$ are continuous in $D$.
\item $A$ satisfies the \emph{unique continuation property}, that is, if $u$ satisfies $Au=0$ in an open connected set $D$ and $u=0$ in some open nonempty $D'\subset D$, then $u=0$ in $D$.
\item For any hyperplane $\Pi$, with unit normal $\nu$, let $D^+\subset H^+_{\Pi}$ be an open set and let $u^+$ satisfy $Au^+=0$ in $D^+$. Then there exists an operator $\tilde{T}_{\Pi,\nu}$ such that
$$u^-=\tilde{T}_{\Pi,\nu}(u^+\circ T_{\Pi})$$
satisfies $Au^-=0$ in $D^-=T_{\Pi}(D^+)$. Furthermore, the following properties are assumed.
\begin{enumerate}
\item Let $x\in \Pi$ be such that $B_r(x)\cap H^+_{\Pi}\subset D^+$. If $Bu^+=0$ on $B_r(x)\cap \Pi$, then $Bu^-=0$ on $B_r(x)\cap \Pi$ as well and, calling
$$u=\left\{\begin{array}{ll}
u^+ \ & \text{in } B_r(x)\cap H^+_{\Pi}\\
u^- \ & \text{in } B_r(x)\cap H^-_{\Pi}
\end{array}
\right.$$
we have that $Au=0$ in $B_r(x)$.
\item On the other hand, if $Au=0$ in $B_r(x)$ and
$$u|_{B_r(x)\cap H^-_{\Pi}}=\tilde{T}_{\Pi,\nu}(u|_{B_r(x)\cap H^+_{\Pi}}\circ T_{\Pi}),$$ 
then $Bu=0$ on $B_r(x)\cap \Pi$ (on either sides of $\Pi$).
\item Let $\Pi_1$ be any hyperplane and let $x\in \Pi_1$ be such that, for some $\nu_1$ orthogonal to $\Pi_1$, $B_r(x)\cap H^+_{\Pi_1} \subset D^+$. If $Bu^+=0$ on $B_r(x)\cap \Pi_1$, then
$Bu^-=0$ on $T_{\Pi}(B_r(x)\cap \Pi_1)$.
\end{enumerate}
\end{enumerate}
\end{assumption}

\begin{remark}\label{ossprel1}
Important consequences of (A.3)(a) are the following. First, the condition $Bu=0$ is independent on the choice of $\nu$, thus excluding Robin type boundary conditions.
Moreover, let $\Pi$ be a hyperplane and let $x\in \Pi$.
If, for some $r>0$,  $Au=0$  in $B_r(x)$ and $Bu=0$ on $B_r(x)\cap\Pi$, then, by the unique continuation of (A.2),
$u|_{B_r(x)\cap H^-_{\Pi}}=\tilde{T}_{\Pi,\nu}(u|_{B_r(x)\cap H^+_{\Pi}}\circ T_{\Pi})$ and
$u|_{B_r(x)\cap H^+_{\Pi}}=\tilde{T}_{\Pi,-\nu}(u|_{B_r(x)\cap H^-_{\Pi}}\circ T_{\Pi})$.
\end{remark}

\begin{remark}\label{ossprel2}
Let $Au=0$ in an open set $D$. Let $x\in D$ and $\Pi$ a hyperplane passing through $x$. If, for some $r>0$, $B_r(x)\subset D$ and $Bu=0$ on $B_r(x)\cap\Pi$, then
$Bu=0$ on $S$, the connected component of $D\cap \Pi$ containing $x$. In fact, let $S_1$ be the largest open subset of $S$ where $Bu=0$. By contradiction, let us assume that $S_1$ is different from $S$, that is, there exists $x_1\in S$ belonging to the boundary of $S_1$. We assume that, for some $r_1>0$, $B_{r_1}(x_1)\subset D$, therefore $B_{r_1}(x_1)\cap \Pi\subset S$. Moreover, there exists $x_2\in S_1$ and $r_2>0$ such that
$B_{r_2}(x_2)\subset B_{r_1}(x_1)$ and $Bu=0$ on $B_{r_2}(x_2)\cap\Pi$.
By the previous remark,
$u|_{B_{r_2}(x_2)\cap H^-_{\Pi}}=\tilde{T}_{\Pi,\nu}(u|_{B_{r_2}(x_2)\cap H^+_{\Pi}}\circ T_{\Pi})$ and, 
again by the unique continuation of (A.2), the same holds on $B_{r_1}(x_1)$. By (A.3)(b), we conclude that $Bu=0$ on $B_{r_1}(x_1)$ and we obtain a contradiction.
\end{remark}

 \begin{definition}
We say that $\Sigma\subset \R^N$ is a \emph{scatterer} if it is bounded, closed and $G=\mathbb{R}^N\backslash\Sigma$ is connected. In turn, $G$ is called an \emph{exterior domain}, that is, a connected open set whose complement is bounded.

We say that a scatterer $\Sigma$ is an \emph{obstacle} if $\Sigma$ coincides with the closure of its interior. 

We call \emph{cell} the closure of an open subset of an $(N-1)$-dimensional hyperplane.

We say that a scatterer $\Sigma$ is \emph{polyhedral} if its boundary is the finite union of cells. Without loss of generality we assume that different cells are internally disjoint.
\end{definition}

\begin{remark} We note that a cell need not to be an $(N-1)$-dimensional polyhedron contained in an $(N-1)$-dimensional hyperplane. Moreover,
its relative interior can have infinitely many connected components.

However, if $\Sigma$ is a polyhedral obstacle, then it is just the union of a finite number of pairwise internally disjoint $N$-dimensional polyhedra.
Hence, we can decompose the boundary of a polyhedral obstacles into a finite number of 
$(N-1)$-dimensional polyhedra contained in an $(N-1)$-dimensional hyperplane. We call these $(N-1)$-dimensional polyhedra the $(N-1)$-\emph{faces} of $\Sigma$ and we denote with $\mathcal{F}_{N-1}$ the set of all $(N-1)$-faces of $\Sigma$.
In turn, the boundary of any $(N-1)$-face of $\Sigma$ is composed by a finite number of 
$(N-2)$-dimensional polyhedra contained in an $(N-2)$-dimensional subspace. We call these $(N-2)$-dimensional polyhedra the $(N-2)$-\emph{faces} of $\Sigma$
and we denote with $\mathcal{F}_{N-2}$ the set of all $(N-2)$-faces of $\Sigma$.
Proceeding iteratively, we can define $k$-\emph{faces} of $\Sigma$ for $0\leq k\leq N-1$ and we denote with $\mathcal{F}_{k}$ the set of all $k$-faces of $\Sigma$.
We usually call \emph{vertices} of $\Sigma$ the $0$-faces of $\Sigma$. An important property to recall is that if $P$ is a vertex of $\Sigma$, then the span of the normals to all $(N-1)$-faces of $\Sigma$ containing $P$ is the whole $\R^N$.
 \end{remark}

Let $\Sigma$ be a polyhedral scatterer. We assume that $u$ is a solution to $Au=0$ in $G=\R^N\backslash\Sigma$ and that $Bu=0$ on flat portions of boundary. That is, for any $x\in \partial \Sigma$, if $x$ belongs to the interior of a cell $C$ contained in a hyperplane $\Pi$ and, for some $\nu$ orthogonal to $\Pi$ and some $r>0$, we have $B_r(x)\cap H^+_{\Pi}\subset G$, then $Bu=0$ on $B_r(x)\cap \Pi$. Note that if also $B_r(x)\cap H^-_{\Pi}\subset G$, then $Bu=0$ on either sides of $\Pi$.
We can also state this condition as follows. For any cell $C$, contained in a hyperplane $\Pi$, and any $\nu$ orthogonal to $\Pi$ such that for any 
$x$ belonging to the interior of $C$ there exists $r>0$ with $B_r(x)\cap H^+_{\Pi}\subset G$, we have $Bu=0$ in the interior of the cell $C$.

Following \cite{Ale-Ron}, we consider the following.
 
 \begin{definition}\label{flatdef}
 A point $x\in G$ is a \emph{flat point} for $u$ if there exists $\Pi$ passing through $x$ such that $Bu=0$ on $\Pi\cap B_r(x)$ for some $r>0$ with $B_r(x)\subset G$.
 
 For any flat point $x\in G$, 
 let $S$ be the connected component of $G\cap \Pi$ containing $x$. By Remark~\ref{ossprel2}, we have that $Bu=0$ on $S$. Fixing $\nu$ orthogonal to $\Pi$, we call $G^+$ the connected component of $G\backslash S$ containing $B_r(x)\cap H^+_{\Pi}$ and
 $G^-$ the connected component of $G\backslash S$ containing $B_r(x)\cap H^-_{\Pi}$. We note that it may happen that $G^+=G^-$.
 We call $E^+$ the connected component of $G^+\cap T_{\Pi}(G^-)$ containing $B_r(x)\cap H^+_{\Pi}$
 and $E^-$ the connected component of $G^-\cap T_{\Pi}(G^+)$ containing $B_r(x)\cap H^-_{\Pi}$. We note that $E^-=T_{\Pi}(E^+)$ and that
 $u|_{E^-}=\tilde{T}_{\Pi,\nu}(u|_{E^+}\circ T_{\Pi})$.
 Let $E=E(x)$ be the set $E=E^+\cup E^-\cup S$.
We note that $E$ is a connected set contained in $G$ which is symmetric with respect to $\Pi$. Moreover, the boundary of $E$ is bounded and contained in $\partial\Sigma\cup T_{\Pi}(\partial\Sigma)$ and any point $y\in \partial E\backslash\Sigma$ is a flat point.

Finally, any hyperplane $\Pi$ such that $Bu=0$ on $\Pi\backslash \overline{B_R}$ for some $R>0$ will be called a \emph{reflection hyperplane}. 

\end{definition}
 
 \begin{remark}\label{remark1} Let $x\in G$ be a flat point, with respect to the hyperplane $\Pi$, and let $E=E(x)$. Then there are two cases: either $E$ is unbounded or $E$ is bounded. If $E$ is unbounded, then there exists $R>0$ such that
$\Pi\backslash \overline{B_R}$ is contained in $E$, hence, by (A.3)(b), $Bu=0$ on $\Pi\backslash \overline{B_R}$. Consequently, if $E$ is unbounded, then $\Pi$ is a reflection hyperplane.

Moreover, if $\Sigma\subset B_{R_0}$, for some $R_0>0$, and $\|x\|>R_0+1$, then $E$ is unbounded, thus $\Pi$ is a reflection hyperplane.
\end{remark}
 
\begin{remark}\label{remark2}
Another important observation is that the reflection in a reflection hyperplane of a reflection hyperplane is still a reflection hyperplane.

In particular, let $\kappa$ be any $(N-2)$-dimensional subspace.
We consider the set of all reflection hyperplanes containing $\kappa$, assuming such a set is not empty. There are two cases.
If the reflection hyperplanes containing $\kappa$ are finite, let us say $m\geq 1$, then they subdivide the whole space into $2m$ equal sectors.
If they are infinite, then, by a simple continuity argument due to the regularity of (A.1), any hyperplane containing $\kappa$ is a reflection hyperplane.
\end{remark}

The key result is the following.
 
 \begin{theorem}\label{thm1}
Let us assume that $u$ admits a flat point $x\in G$. Then there exists a reflection hyperplane $\Pi_1$.
 \end{theorem}
 
 \begin{proof}
 Let $\gamma:[0,+\infty)\to \R^N$ be a piecewise smooth curve in $G$ with first endpoint $\gamma(0)=x$ and such that $\|\gamma(t)\|\to+\infty$ as $t\to+\infty$.
 
 First of all, we claim that the set of flat points belonging to $\gamma$ is closed. Analogously, the set of $t\in [0,+\infty)$ such that $x=\gamma(t)$ is a flat point is closed.
  In fact, let $x_n=\gamma(t_n)$, $n\in\N$, be a sequence of flat points converging to $x=\gamma(t)$ as $n\to +\infty$. We need to show that $x$ is a flat point as well. Let $\Pi_n$ be the hyperplane related to $x_n$ and $\nu_n$ one of its unit normal. Up to a subsequence, we can assume that $\nu_n\to \nu$ as $n\to+\infty$ and we call $\Pi$ the hyperplane orthogonal to $\nu$ passing through $x$. A simple continuity argument due to the regularity of (A.1) allows us to prove the claim.
 
 If the set $A=\{t\in[0,+\infty):\ \gamma(t)\text{ is a flat point}\}$ is unbounded, then the proof is concluded by the last part of Remark~\ref{remark1}.
 
 Assume that $A$ is bounded and let $\tilde{t}=\max A$ and $\tilde{x}=\gamma(\tilde{t})$. If the corresponding set $E=E(\tilde{x})$ is unbounded, the proof is concluded.
 It remains the case in which $E$ is bounded. Then there exists $\tilde{t}_1>\tilde{t}$ such that $\tilde{x}_1=\gamma(\tilde{t}_1)$ belongs to $\partial E\cap G=\partial E\backslash\Sigma$. Hence $\tilde{x}_1$ is a flat point and we obtain a contradiction. 
 Thus the proof is concluded.
\end{proof}
 
 When we are dealing with polyhedral obstacles, we can prove something more.
 
 \begin{theorem}\label{thm2}
Let $\Sigma$ be a polyhedral obstacle. Let $\Sigma'$ be a finite union of polyhedra with $\Sigma'\backslash\Sigma\neq\emptyset$.
Let $\tilde{G}$ be the unbounded connected component of $\R^N\backslash(\Sigma\cup\Sigma')$.
Assume that any $x\in\partial \tilde{G}\cap G$ is a flat point.
 
 Then there exist $N$ reflection hyperplanes $\Pi_1,\ldots,\Pi_N$, whose corresponding normals $\nu_1,\ldots,\nu_N$ are linearly independent.
 \end{theorem}
 
 In order to prove this result we need several preliminary facts and lemmas. First of all we consider some easy cases.
 
 If $\Pi$ is a reflection hyperplane such that
 $\Sigma\subset H^+_{\Pi}$ for some $\nu$ orthogonal to $\Pi$, then the result is proved by the following argument. 
 Since $\Sigma$ is an obstacle, we can find $(N-1)$-faces of $\Sigma$, $C_1,\ldots,C_N$, belonging to hyperplanes $\tilde{\Pi}_1,\ldots,\tilde{\Pi}_N$ with normals $\tilde{\nu}_1,\ldots,\tilde{\nu}_N$ respectively, such that $\tilde{\nu}_1,\ldots,\tilde{\nu}_N$ are linearly independent.
By reflecting in $\Pi$, it is easy to see that the hyperplanes $\Pi_i=T_{\Pi}(\tilde{\Pi}_i)$, $i=1,\ldots,N$, satisfy the thesis.

If there exist two different reflection hyperplanes with the same normal $\nu$, by iterative reflections we can find infinitely many equispaced reflection hyperplanes with the same normal. Eventually, for one of these, let us call it $\Pi$, we have $\Sigma\subset H^+_{\Pi}$, and the proof is concluded.

Let $R_0>0$ be such that $\Sigma$ is contained in $B_{R_0}$.
Let $\kappa$ be an $(N-2)$-dimensional subspace whose distance from the origin is greater than or equal to $2R_0+1$. Let $Q\in\kappa$ be the projection of the origin on
$\kappa$. Let $\nu_0=Q/\|Q\|$. Then there exists a constant $\alpha_0$, $0<\alpha_0<1$, depending on $R_0$ only, such that if $\Pi$ is a reflection hyperplane containing $\kappa$ and with normal $\nu$ such that
$\nu\cdot\nu_0\geq \alpha_0$, then $\overline{B_{R_0}}$, hence $\Sigma$, is contained in $H^+_{\Pi}$ and the proof is concluded. In particular, assume
we have
two different reflection hyperplanes $\Pi_1$ and $\Pi_2$ whose intersection is $\kappa$ and whose normals $\nu_1$ and $\nu_2$ satisfy $\alpha_0\leq \nu_1\cdot\nu_2<1$.
We have that $\nu_0$ belongs to the plane spanned by $\nu_1$ and $\nu_2$. By Remark~\ref{remark2}, we can find a reflection hyperplane $\Pi$ containing $\kappa$ such that its normal $\nu$ satisfies $\nu\cdot\nu_0\geq \alpha_0$, and the proof is concluded.

From now on, let $N_0$, $1\leq N_0\leq N$, be the dimension of the span of the normals to reflection hyperplanes.

\begin{lemma}\label{lemma21}
If $N_0<N$, then the number of reflection hyperplanes is finite.
\end{lemma}

\begin{proof}
Assume by contradiction that this is not the case. We can
find a sequence of different reflection hyperplanes $\Pi_n$, with normal $\nu_n$, such that $\nu_n$ converges to $\nu_{\infty}$ as $n\to+\infty$. For $m>n$, let $\kappa_{n,m}=\Pi_n\cap\Pi_m,$
 and $\alpha_{n,m}=\nu_n\cdot\nu_m$.  Without loss of generality, we can assume that $\alpha_0\leq \alpha_{n,m}<1$ for any $m>n$.
 If for some $m>n$ the distance of $\kappa_{n,m}$ from the origin is greater than or equal to $2R_0+1$, we obtain a contradiction by the property described above.
 Hence, we can assume that the distance of $\kappa_{n,m}$ from the origin is less than $2R_0+1$  for any $m>n$.
 
Let $\mathcal{F}=\bigcup_{k=0}^{N-1}\mathcal{F}_k$ be the set of all faces of $\Sigma$. Let $\delta>0$ be the minimum distance between any two disjoint faces of $\Sigma$, that is,
$$\delta=\min\left\{\mathrm{dist}(\sigma,\sigma'):\ \sigma,\,\sigma'\in\mathcal{F},\ \sigma\cap\sigma'=\emptyset\right\}.$$

We can find a constant $\alpha_1$, $\alpha_0\leq\alpha_1<1$, depending on $R_0$ and $\delta$ only, such that if $\alpha_1\leq \alpha_{n,m}<1$ the following holds.
For any $P\in\overline{B_{R_0}}$, we can find a reflection hyperplane $\Pi$ containing $\kappa_{n,m}$, depending on $P$, such that if $P'$ is the reflection of $P$ in $\Pi$, then $0<\|P'-P\|\leq \delta/2$. Again without loss of generality, we can assume that $\alpha_1\leq \alpha_{n,m}<1$ for any $m>n$.

Let us fix any $m>n$. We have that $\alpha=\alpha_{n,m}$ satisfies $\alpha_1\leq \alpha<1$. We call $\kappa=\kappa_{n,m}$.
Without loss of generality, up to a rigid change of coordinates, let $\kappa=\tilde{Q}\times \R^{N-2}$, with $\tilde{Q}$ a point in $\R^2$, so that $Q=(\tilde{Q},0,\ldots,0)$
is the element of $\kappa$ closest to the origin.
Let, for any $P=(P_1,\ldots,P_N)\in \R^N$, $s(P)$ denote the distance of $P$ from $\kappa$.
For any $r>0$, let $\kappa_r=\{P\in \R^N:\ s(P)\leq r\}$. We begin with the following remark.
Let $r_0=\min\{r:\ \Sigma\subset \kappa_r\}$.
If $P\in\Sigma$ is such that $s(P)=r_0$, then $P$
can not belong to the interior of an $(N-1)$-face of $\Sigma$. Instead, it
 belongs to a $k$-face of $\Sigma$, with $k\leq N-2$, contained in $(P_1,P_2)\times \R^{N-2}$. Note that such a face may consist of $P$ only.
 Actually, we can indeed find a vertex $P$ of $\Sigma$ such that $s(P)=r_0$. Let $P'$ be its reflection in a reflection hyperplane $\Pi$ as constructed above. We have that
 $s(P')=r_0$, $P'$ is not contained in $(P_1,P_2)\times \R^{N-2}$ and its distance from $P$ is less than $\delta$. Hence $P'$ can not belong to $\Sigma$. On the other hand, we can find $(N-1)$-faces of $\Sigma$, $C_1,\ldots,C_N$, belonging to hyperplanes $\tilde{\Pi}_1,\ldots,\tilde{\Pi}_N$ with normals $\tilde{\nu}_1,\ldots,\tilde{\nu}_N$ respectively, such that $\tilde{\nu}_1,\ldots,\tilde{\nu}_N$ are linearly independent and $P\in C_i$ for any $i=1,\ldots,N$.
By reflecting in $\Pi$, it is easy to see that the hyperplanes $\Pi_i=T_{\Pi}(\tilde{\Pi}_i)$, $i=1,\ldots,N$, are actually reflection hyperplanes passing through $P'$ and thus we obtain a contradiction. The proof is concluded.
\end{proof}

Before stating the next lemma, we recall the following classical result, see for details \cite[Lemma~2.14]{LRX}.

\begin{remark}\label{connremark}
Let $D$ be an open connected set in $\R^N$, $N\geq 2$. Let $C$ be the union of a finite number of $(N-2)$-dimensional subspaces. Then $D\backslash C$ is still connected.
\end{remark}

\begin{lemma}\label{lemmasimplex}
Let $N_0<N$.
Let $S$ be a linear subspace of $\R^N$  of dimension $M$, with $1< M\leq N_0$. Then there can not exist $M+1$ reflection hyperplanes $\Pi_1,\ldots,\Pi_{M+1}$, with normals $\nu_1,\ldots,\nu_{M+1}$ belonging to $S$, and $H_0$, an $M$-dimensional bounded convex polyhedron in $S$ whose boundary is contained in the union of $\tilde{\Pi}_i=\Pi_i\cap S$, $i=1,\ldots,M+1$.
\end{lemma}

\begin{proof} We let $\R^N=\R^M\times \R^{N-M}$ and, without loss of generality, $S=\R^M\times \{0\}$. We identify $S$ with $\R^M$, $H_0$ with a bounded convex polyhedron in $\R^M$ and $\tilde{\Pi}_i$, $i=1,\ldots,M+1$, with hyperplanes in $\R^M$. Let us consider all reflection hyperplanes $\Pi$
of the kind $\Pi=\tilde{\Pi}\times \R^{N-M}$ with $\tilde{\Pi}$ a hyperplane in $\R^M$. With a little abuse of notation, we still call $\tilde{\Pi}$ a reflection hyperplane.

Let $K$ be the union of all $(M-2)$-dimensional subspaces of $\R^M$ obtained by intersecting
any pair of different reflection hyperplanes $\tilde{\Pi}$.
Let $V$ be the set of all points of $\R^M$ obtained by intersecting any $M$ reflection hyperplanes whose normals are linearly independent.
 The set $K$ is closed whereas the set $V$ is finite, hence bounded and closed. 
Let us fix $R>0$ such that $V\subset B_R$. Consequently, $H_0\subset B_R$ as well.

Let $\gamma:[0,1]\to\R^M$ be a smooth curve connecting $x_0$, a point in the interior of $H_0$, to $x_1$, a point such that $\|x_1\|\geq 4R+2$. We can assume that $x_1$ does not
belong to any reflection hyperplane $\tilde{\Pi}$ and, by Remark~\ref{connremark}, that $\gamma$ does not contain any element of $K$. Moreover, we can assume that for some $c>0$, $\|\gamma'(t)\|\geq c$ for any $t\in [0,1]$.

Let $t_0$ be the largest $t\in [0,1]$ such that $\gamma(t)\in H_0$. Let $\tilde{\Pi}^0$ be the (only) reflection hyperplane to which $\gamma(t_0)$ belongs. We call $H_1$ the reflection of
$H_0$ in $\tilde{\Pi}^0$. It is clear that for some $t>t_0$, $\gamma(t)$ belongs to the interior of $H_1$. Moreover, any $(M-1)$-dimensional face of $H_1$ is contained in a reflection hyperplane and any vertex of $H_1$ belongs to $V$.
Let $t_1>t_0$ be the last $t\in (t_0,1]$ such that $\gamma(t)\in H_1$. Let $\tilde{\Pi}^1$ be the (only) reflection hyperplane to which $\gamma(t_1)$ belongs.
We note that $\tilde{\Pi}^1$ is different from $\tilde{\Pi}^0$.
 We call $H_2$ the reflection of
$H_1$ in $\tilde{\Pi}^1$. It is clear that for some $t>t_1$, $\gamma(t)$ belongs to the interior of $H_2$. Moreover, any $(M-1)$-dimensional face of $H_2$ is contained in a reflection hyperplane and any vertex of $H_2$ belongs to $V$. By induction, with the same procedure, we construct $t_1,\ldots,t_n$ until $\|\gamma(t_n)\|>4R+1$. If this is the case, $H_n\cap B_R=\emptyset$ but the vertices of $H_n$ must belong to $V$ and we have a contradiction. Otherwise, we construct a sequence $\{t_n\}_{n\in\N}$, with
$t_n\to t_{\infty}<1$. Since $t_{n+1}-t_n\to 0$ as $n\to+\infty$, and $\tilde{\Pi}^{n+1}$ is different from $\tilde{\Pi}^{n}$ for any $n\in\N$,
we can find a point $P_n\in H_n\cap K$ such that $\|\gamma(t_n)-P_n\|\to 0$ as well. Clearly, $P_n\to\gamma(t_{\infty})$ as $n\to+\infty$. Since $K$ is closed, we conclude that $\gamma(t_{\infty})\in K$, thus we obtain a contradiction. 
\end{proof}

\begin{lemma}\label{lemmacommonplane}
If $N_0<N$, then there exists an $(N-N_0)$-dimensional subspace $\kappa$ such that $\kappa$ is contained in all reflection hyperplanes.

Let $\Pi_1,\ldots,\Pi_m$ be all the reflection hyperplanes and $W=\{T_{\Pi_1},\ldots,T_{\Pi_m}\}$. Then the set
$$\mathcal{T}=\big\{T:\ T\text{ is the finite composition of reflections belonging to }W\big\}$$
is finite.
\end{lemma}

\begin{proof}
Let $\Pi_1,\ldots,\Pi_{N_0}$ be reflection hyperplanes whose normals are linearly independent. 
Then let $\kappa$ be the intersection of $\Pi_1,\ldots,\Pi_{N_0}$. Clearly
$\kappa$ is
an $(N-N_0)$-dimensional subspace. We need to show that all reflection hyperplanes contain $\kappa$.

Let $\R^N=\R^{N_0}\times \R^{N-N_0}$.
For simplicity and without loss of generality, let us assume that the span of the corresponding normals
$\nu_1,\ldots,\nu_{N_0}$ is $S_{N_0}=\R^{N_0}\times\{0\}$ and that $\kappa=\{0\}\times \R^{N-N_0}$.
It is enough to consider the geometry in $S_{N_0}$. For any reflection hyperplane $\Pi$, we call $\tilde{\Pi}$ its intersection with $S_{N_0}$ and we identify it with a hyperplane of $\R^{N_0}$, which with a little abuse of notation we still call $\tilde{\Pi}$ and still refer to as a reflection hyperplane.

Let us assume, by contradiction, that there exists a reflection hyperplane $\tilde{\Pi}_{N_0+1}$ not passing through the origin. Let $\nu_{N_0+1}$ be its normal.
If any $N_0$ elements of $\{\nu_1,\ldots,\nu_{N_0+1}\}$ are linearly independent, then $\tilde{\Pi}_1,\ldots,\tilde{\Pi}_{N_0+1}$ bound a convex polyhedron thus contradicting Lemma~\ref{lemmasimplex} with $M=N_0$.

Therefore, we can assume, up to reordering, that $\nu_2,\ldots,\nu_{N_0+1}$ are linearly dependent. Let $S_{N_0-1}$ be the span of $\nu_2,\ldots,\nu_{N_0}$. Let us call
$\tilde{\Pi}_i^1$ the restrictions of $\Pi_i$ to $S_{N_0-1}$. The intersection of $\tilde{\Pi}_2^1,\ldots,\tilde{\Pi}_{N_0}^1$ is the origin of
$S_{N_0-1}$ and $\tilde{\Pi}_{N_0+1}^1$ can not pass through the origin of $S_{N_0-1}$, otherwise $\tilde{\Pi}_{N_0+1}$ would pass through the origin of $S_{N_0}$.
If any $N_0-1$ elements of $\{\nu_2,\ldots,\nu_{N_0+1}\}$ are linearly independent, then $\tilde{\Pi}_2^1,\ldots,\tilde{\Pi}_{N_0+1}^1$ bound a convex polyhedron thus contradicting Lemma~\ref{lemmasimplex} with $M=N_0-1$.

Iterating the procedure, and with suitable reordering, we obtain that $\nu_{N_0-1},\nu_{N_0},\nu_{N_0+1}$ are linearly dependent. Calling $S_2$ the span of 
$\nu_{N_0-1}$ and $\nu_{N_0}$, and $\tilde{\Pi}_i^{N_0-2}$ the restrictions of $\tilde{\Pi}_i$ to $S_2$, we obtain that
the intersection of $\tilde{\Pi}_{N_0-1}^{N_0-2}$ and $\tilde{\Pi}_{N_0}^{N_0-2}$ is the origin of
$S_2$ and $\tilde{\Pi}_{N_0+1}^{N_0-2}$ can not pass through the origin of $S_2$. If any pair of
$\{\nu_{N_0-1},\nu_{N_0-1},\nu_{N_0+1}\}$ are linearly independent, we contradict Lemma~\ref{lemmasimplex} with $M=2$.
Hence $\Pi_{N_0+1}$ is parallel to, and different from, either $\Pi_{N_0-1}$ or $\Pi_{N_0}$. This contradicts the fact that $N_0<N$ and the first part of the proof is concluded.

Once we have shown that the all reflection hyperplanes intersect $\kappa$, we have that any element of $\mathcal{T}$ is a rigid change of coordinates keeping $\kappa$ fixed. Let us consider $T\in \mathcal{T}$. The set $\{T(\nu_1),\ldots,T(\nu_{N_0})\}$ is a base of $S_{N_0}$ and fully characterises $T$. 
Since the image through $T$ of any reflection hyperplane $\Pi$ is another reflection hyperplane, the set $\{T(\nu_1),\ldots,T(\nu_{N_0})\}$ is a subset of
$\{\pm \nu_1,\ldots,\pm\nu_m\}$. It immediately follows that $\mathcal{T}$ is finite.
\end{proof}

\begin{lemma}\label{pathlemma} Let $N_0<N$.
Let $\Sigma_0$ be a polyhedral scatterer which is symmetric with respect to all reflection hyperplanes. We call $G_0=\R^N\backslash\Sigma_0$.
Let $x_0$ be any point in $G_0$ not belonging to any reflection hyperplane. 
Then there exists a piecewise smooth path $\gamma:[0,+\infty)\to\R^N$ such that
$\|\gamma(t)\|\to+\infty$ as $t\to+\infty$ and $\gamma(t)$ belongs to $G_0$ and does not belong to any reflection hyperplane for any $t\geq 0$.
\end{lemma}

\begin{proof} Let $\Pi_1,\ldots,\Pi_m$ be all reflection hyperplanes and $\kappa$ be their intersection, as in Lemma~\ref{lemmacommonplane}.

We can assume that $0\in\kappa$ and that, for some $R>0$, we have
$\Sigma\subset B_R$. Let $D_0$ be the connected component of $\R^N\backslash \left(\bigcup_{i=1}^m\Pi_i\right)$ containing $x_0$. Note that for any $s\geq R+1$, $\partial B_s\cap D_0$ is contained in $G_0$ and is connected.

Let $P$ be the projection of $x_0$ on $\kappa$. Let $l$ be the half-line starting from $P$ and passing through $x_0$. With the exception of $P$, any element of $l$ belongs to $D_0$. Let $x_1$ be an element of $l$ such that after $x_1$ any element of $l$ is outside $B_{R+1}$. Hence the half-line contained in $l$ which starts from $x_1$ is contained in $G_0$ and does not intersect
any reflection hyperplane. In order to conclude the proof, it is enough to show that $D_0\cap G_0$ is connected, namely that there exists a piecewise smooth curve $\gamma$ in $D_0\cap G_0$ connecting $x_0$ to $x_1$. We denote $s=\|x_1\|\geq R+1$.

 Let $K$ be the union of the intersections of any pair of different reflection hyperplanes. By Remark~\ref{connremark}, $G_0\backslash K$ is still connected. Let $\gamma_0:[0,1]\to \R^N$ be a smooth curve in $G_0\backslash K$ connecting $x_0$ to $x_1$.
 We can assume that for some constant $c>0$ we have $\|\gamma_0'(t)\|\geq c$ for any $t\in [0,1]$.
 
 If $\gamma_0\subset \overline{D_0}$ there is nothing to prove. In fact, by a small perturbation argument we can change $\gamma_0$ such that $\gamma_0\subset G_0\cap D_0$. Otherwise, calling $t'_0=0$, let
 $t_1>0 $ be largest $t$ 
 such that $\gamma_0(t_1)\in\partial D_0$ and $\gamma_0(t)\in \overline{D_0}$ for any $0\leq t<t_1$. We call $\tilde{\Pi}_1$ the reflection hyperplane containing $\gamma_0(t_1)$ and 
 $D_1$ the reflection of $D_0$ in $\tilde{\Pi}_1$. Let $t'_1\geq t_1$ be the largest $t$ such that $\gamma_0(t'_1)\in \tilde{\Pi}_1$ and $\gamma_0(t)\in \overline{D_0}\cup D_1$ for any $t\in
 (t_1,t'_1)$. By the symmetry of $\Sigma_0$,
$T_{\tilde{\Pi}_1}\big(\gamma_0([t_1,t'_1]\cap D_1)\big)$ is contained in $G_0\cap D_0$. Hence, by performing such a reflection, we can assume that $\gamma_0(t)\in \overline{D_0}$ for any $t\in [0,t'_1)$. Let us define for any $t\in [0,1]$
$$\gamma_1(t)=\left\{
\begin{array}{ll}\gamma_0(t) &\text{if }t\in [0,t'_1]\\
T_{\tilde{\Pi}_1}(\gamma_0(t)) & \text{if }t\in (t'_1,1].
\end{array}
\right.$$

If $\gamma_1\subset \overline{D_0}$, by a small perturbation argument we can change $\gamma_1$ such that $\gamma_1\subset G_0\cap D_0$. Then we connect $\gamma_1(1)$ to $x_1$ along $\partial B_s\cap D_0$ and the result is proved.
Otherwise, let
 $t_2>t'_1$ be largest $t$ 
 such that $\gamma_1(t_2)\in\partial D_0$ and $\gamma_1(t)\in \overline{D_0}$ for any $0\leq t<t_2$. We call $\tilde{\Pi}_2$ the reflection hyperplane containing $\gamma_1(t_2)$ and 
 $D_2$ the reflection of $D_0$ in $\tilde{\Pi}_2$.
 Note that $\tilde{\Pi}_2$ is different from $\tilde{\Pi}_1$ by construction.
  Let $t'_2\geq t_2$ be the largest $t$ such that $\gamma_1(t'_2)\in \tilde{\Pi}_2$ and $\gamma_1(t)\in \overline{D_0}\cup D_2$ for any $t\in
 (t_2,t'_2)$. By the symmetry of $\Sigma_0$,
$T_{\tilde{\Pi}_2}\big(\gamma_1([t_2,t'_2]\cap D_2)\big)$ is contained in $G_0\cap D_0$. Hence, by performing such a reflection, we can assume that $\gamma_1(t)\in \overline{D_0}$ for any $t\in [0,t'_2)$. Let us define for any $t\in [0,1]$
$$\gamma_2(t)=\left\{
\begin{array}{ll}\gamma_1(t) &\text{if }t\in [0,t'_2]\\
T_{\tilde{\Pi}_2}(\gamma_1(t)) & \text{if }t\in (t'_2,1].
\end{array}
\right.$$

We proceed iteratively in the same way. If, for some $n\in \N$, $\gamma_n\subset \overline{D_0}$, by a small perturbation argument we can change $\gamma_n$ such that $\gamma_n\subset G_0\cap D_0$. Then we connect $\gamma_n(1)$ to $x_1$ along $\partial B_s\cap D_0$ and the result is proved. Otherwise, we find a 
sequence $\{t'_n\}_{n\in\N}$ with $0<t'_n<t'_{n+1}<1$ such that $\gamma_{n-1}(t'_n)\in\partial D_0$, with $\tilde{\Pi}_n$ being the reflection hyperplane containing $\gamma_{n-1}(t'_n)$, and $\gamma_{n-1}(t)\in \overline{D_0}$ for any $0\leq t\leq t'_n$. We note that, by construction, $\tilde{\Pi}_{n+1}$ is different from $\tilde{\Pi}_n$ for any $n\in\N$. Let $t'_{\infty}=\lim_n t'_n$. We have that $\gamma_0(t'_n)\in \bigcup_{i=1}^m\Pi_i$, hence $\gamma_0(t'_{\infty})\in \bigcup_{i=1}^m\Pi_i$ as well. We conclude that
$t'_{\infty}<1$. Let $\hat{\Pi}_n$ be the reflection hyperplane to which $\gamma_0(t'_n)$ belongs. We note that
$\hat{\Pi}_{n+1}$ is different from $\hat{\Pi}_n$ for any $n\in\N$.
Since $t'_{n+1}-t'_n\to 0$ as $n\to+\infty$, we can find a point $P_n\in \hat{\Pi}_{n}\cap \hat{\Pi}_{n+1}\subset K$ such that $\|\gamma_0(t'_n)-P_n\|\to 0$ as well. Clearly, $P_n\to\gamma_0(t'_{\infty})$ as $n\to+\infty$. Since $K$ is closed, we conclude that $\gamma_0(t'_{\infty})\in K$, thus we obtain a contradiction. 
\end{proof}

\begin{lemma}\label{symmetriclemma}
If $N_0<N$, then $\Sigma$ is symmetric with respect to all reflection hyperplanes.
\end{lemma}

\begin{proof}
Let us consider $\kappa$ and the set $\mathcal{T}$ defined in Lemma~\ref{lemmacommonplane}. Without loss of generality, we assume that $0\in\kappa$ and that for some $R>0$ we have that $\Sigma\subset B_R$.

For any $T\in \mathcal{T}$, we can define $u_T$ in $T(G)=\R^N\backslash T(\Sigma)$ such that $Au_T=0$ in $T(G)$ and $u_T=u$ outside $B_{R+1}$.
 Moreover, $Bu_T=0$ on any flat part of $\partial (T(\Sigma))$.
 
 Let $G_0$ be the unbounded connected component of $\bigcap_{T\in\mathcal{T}}T(G)$ and $\Sigma_0=\R^N\backslash G_0$.
 Let us show that
$G_0$ is symmetric with respect to any reflection hyperplane.
Let $x\in G_0$, $\Pi$ be a reflection hyperplane and
$y=T_{\Pi}(x)$. Let $\gamma:[0,1]\to\R^N$ be a continuous curve in $G_0$ connecting $x$ to a point $x_1$ outside $B_{R+1}$.
For any $t\in [0,1]$, $\gamma(t)\in T(G)$ for any $T\in\mathcal{T}$. Since $T_{\Pi}\circ T$ still belongs to $\mathcal{T}$,
$\gamma(t)\in (T_{\Pi}\circ T)(G)$ for any $T\in\mathcal{T}$, hence
$T_{\Pi}(\gamma(t))\in T_{\Pi}\big((T_{\Pi}\circ T)(G)\big)=T(G)$ for any $T\in\mathcal{T}$. Hence $\gamma_1=T_{\Pi}(\gamma)\subset \bigcap_{T\in\mathcal{T}}T(G)$.
Since $\R^N\backslash \overline{B_R}\subset G_0$
and $y_1=T_{\Pi}(x_1)$ is also outside $B_{R+1}$, we conclude that $y\in G_0$ as well.

On $G_0$, we have that $u=u_T$ for any $T\in\mathcal{T}$. Since $\Sigma_0=\bigcup_{T\in\mathcal{T}}T(\Sigma)$, we have that $\Sigma_0$ is a polyhedral obstacle
containing $\Sigma$ and $Bu=0$ on any flat portion of $\partial\Sigma_0$. Moreover, by symmetry, no flat portion of $\partial\Sigma_0$ is a subset of any reflection hyperplane.

Let us assume there $\Sigma_0$ is different from $\Sigma$. Then there exists $x\in \partial\Sigma_0\cap G$. Without loss of generality, we can assume that $x$ belongs
to a flat portion of $\partial\Sigma_0$ and $x$ does not belong to any reflection hyperplane. For some $\nu$ normal to the cell of $\partial\Sigma_0$ to which $x$ belongs and some $\varepsilon_0>0$, we have that, for any $t\in (0,\varepsilon_0]$, $x+t\nu\in G_0$ and does not belong to any reflection hyperplane. We call $x_0=x+\varepsilon_0\nu$ and, by Lemma~\ref{pathlemma}, we can find a piecewise smooth curve $\gamma:[0,+\infty)\to\R^N$ such that $\gamma(0)=x$, $\gamma(t)=x+t\nu$ for any $t\in (0,\varepsilon_0]$ and $\gamma(t)$ belongs to $G_0$ and
does not belong to any reflection hyperplane for any $t\geq \varepsilon_0$. We note that $x$ is a flat point for $u$.
By the same argument used in the proof of Theorem~\ref{thm1}, we can find $t_1\in [0,+\infty)$ such that $x_1=\gamma(t_1)$ is a flat point belonging to a reflection hyperplane for $u$.
Since no point of $\gamma$ belongs to any reflection hyperplane, we obtain a contradiction.

We conclude that $\Sigma_0=\Sigma$ hence $\Sigma$ is symmetric with respect to any reflection hyperplane.
\end{proof}

\begin{proof}[Proof of Theorem~\textnormal{\ref{thm2}}]
By renaming $\Sigma'=\R^N\backslash\tilde{G}$, without loss of generality, we can assume that $\Sigma'$ is a polyhedral obstacle, with $\Sigma\subset\Sigma'$
and $\Sigma'\backslash\Sigma\neq\emptyset$.
Let us assume by contradiction that $N_0<N$.
Let $x_0\in G$ be such that $x_0$ does not belong to any reflection hyperplane and $x_0$ is contained in the interior of $\Sigma'$. Let $\gamma$ be as in Lemma~\ref{pathlemma} with $\Sigma_0$ replaced by $\Sigma$. It is clear that there exists $t_0>0$ such that $\gamma(t_0)\in \partial\Sigma'$, hence $\gamma(t_0)$ is a flat point for $u$. Again by the same argument used in the proof of Theorem~\ref{thm1}, we can find $t_1\geq t_0$ such that $y_1=\gamma(t_1)$
is a flat point belonging to a reflection hyperplane for $u$.
This contradicts the properties of $\gamma$ and the proof is concluded.
\end{proof}

\subsection{The case of mixed boundary conditions}\label{sec2bis}

Now we briefly consider the case of mixed boundary conditions. Let us assume that we have $n$ different boundary conditions $B^1,\ldots,B^n$ satisfying
(A.3)(a) and (A.3)(b), with $\tilde{T}_{\Pi,\nu}$ clearly depending on $i\in\{1,\ldots,n\}$. We also assume that, for any $i=1,\ldots,n$, the property (A.3)(c) is replaced by
\begin{enumerate}[({A}.3)(c$_1$)]
\item Let $\Pi_1$ be any hyperplane and let $x\in \Pi_1$ be such that, for some $\nu_1$ orthogonal to $\Pi_1$, $B_r(x)\cap H^+_{\Pi_1} \subset D^+$. Then, for any $j=1,\ldots,n$, if $B^ju^+=0$ on $B_r(x)\cap \Pi_1$, then
$B^ju^-=0$ on $T_{\Pi}(B_r(x)\cap \Pi_1)$.
\end{enumerate}

Let $\Sigma$ be a polyhedral scatterer. We assume that $u$ is a solution to $Au=0$ in $G=\R^N\backslash\Sigma$ and that
for any cell $C$, contained in a hyperplane $\Pi$, and any $\nu$ orthogonal to $\Pi$ such that for any 
$x$ belonging to the interior of $C$ there exists $r>0$ with $B_r(x)\cap H^+_{\Pi}\subset G$, we have $B^iu=0$ in the interior of the cell $C$ for some $i\in\{1,\ldots,n\}$ depending on $C$.

We note that, in particular in the case of polyhedra, on any $(N-1)$-face we can have more than one boundary condition, that is, we can split such a face into two or more internally pairwise disjoint cells and on each of them a different boundary condition is satisfied.

Definition~\ref{flatdef} is replaced by the following.
\begin{definition}\label{flatdefmixed}
 A point $x\in G$ is a \emph{flat point} for $u$ if there exists $\Pi$ passing through $x$ and $i\in\{1,\ldots,n\}$ such that $B^iu=0$ on $\Pi\cap B_r(x)$ for some $r>0$ with $B_r(x)\subset G$.
 
 Furthermore, any hyperplane $\Pi$ such that, for some $i\in\{1,\ldots,n\}$, $B^iu=0$ on $\Pi\backslash \overline{B_R}$ for some $R>0$ will be called a \emph{reflection hyperplane}.
\end{definition}

With the same identical proofs, using the suitable reflection depending on the boundary condition, the two main theorems, Theorem~\ref{thm1} and Theorem~\ref{thm2}, can be restated in the case of mixed boundary conditions as follows.
\begin{theorem}\label{thm3}
Let us assume that $u$ admits a flat point $x\in G$. Then there exists a reflection hyperplane $\Pi_1$.

Let $\Sigma$ be a polyhedral obstacle. Let $\Sigma'$ be a finite union of polyhedra with $\Sigma'\backslash\Sigma\neq\emptyset$.
Let $\tilde{G}$ be the unbounded connected component of $\R^N\backslash(\Sigma\cup\Sigma')$.
Assume that any $x\in\partial \tilde{G}\cap G$ is a flat point.
 Then there exist $N$ reflection hyperplanes $\Pi_1,\ldots,\Pi_N$, whose corresponding normals $\nu_1,\ldots,\nu_N$ are linearly independent.
\end{theorem}

\section{Application to uniqueness results for the determination of polyhedral scatterers and examples}\label{sec3}

Let $\Sigma$ and $\Sigma'$ be two polyhedral scatterers contained in $B_{R_0}$. Let $u$ satisfy $Au=0$ in $G$ and be such that $Bu=0$ on any flat part of $\partial\Sigma$. Let $u'$
satisfy $Au'=0$ in $G'=\R^N\backslash\Sigma'$ and be such that $Bu'=0$ on any flat part of $\partial\Sigma'$. Assume that $u=u'$ on an open subset $\tilde{D}$ of $\tilde{G}$, $\tilde{G}$ being the unbounded connected component of $\R^N\backslash (\Sigma\cup\Sigma')$. For example, $\tilde{D}$ can be any open subset of $\R^N\backslash \overline{B_{R_0}}$. Then, by unique continuation $u=u'$ on $\tilde{G}$, thus $u$ and $u'$ have the same reflection hyperplanes.
Our aim is to prove a uniqueness result, that is, to show that $\Sigma=\Sigma'$. In other words, the measurement $u|_{\tilde{D}}$, or any equivalent one, uniquely determines the polyhedral scatterer $\Sigma$. 
The argument is the following. Assume, by contradiction, that $\Sigma\neq \Sigma'$.
Up to swapping $\Sigma$ with $\Sigma'$, we can find $x\in \partial \tilde{G}\backslash \Sigma\subset \partial \Sigma'\backslash\Sigma$, hence $u$ admits a flat point. By using Theorem~\ref{thm1}, we can conclude that $u$, and $u'$ as well, admits a reflection hyperplane.
If we have suitable conditions guaranteeing that $u$, and equivalently $u'$, does not admit any reflection hyperplane,
we obtain a contradiction and the uniqueness result holds true.

Under the same assumptions, if $\Sigma$ and $\Sigma'$ are different polyhedral obstacles, by 
using Theorem~\ref{thm2} instead, we can conclude that $u$, and $u'$ as well, admits $N$ reflection hyperplanes whose normals are linearly independent.
If we have suitable conditions guaranteeing that $u$, and equivalently $u'$, does not admit $N$ reflection hyperplanes whose normals are linearly independent, the measurement $u|_{\tilde{D}}$, or any equivalent one, uniquely determines the polyhedral obstacle $\Sigma$.

For example, let us assume that $u$ and $u'$ can be written as
$$u=u^i+u^s\quad\text{and}\quad u'=u^i+(u')^s$$
where $u^i$ is an entire solution to $Au=0$. About $u^s$, we assume that it satisfies the following decay condition at infinity, namely that, for any $\varepsilon>0$, there exists $R>R_0$ such that, for any unit vector $\nu$,
\begin{equation}\label{conseqofradiation}
|Bu^s|,\,|B((u')^s)|\leq \varepsilon\quad\text{outside }B_R.
\end{equation}
Then on any reflection hyperplane $\Pi$ of $u$ or of $u'$ we have
\begin{equation}\label{propertyofincident}
|Bu^i|\leq \varepsilon\quad \text{on }\Pi\backslash \overline{B_R}.
\end{equation}

If $u=u'$ on $\tilde{D}$, or equivalently $u^s=(u')^s$ on $\tilde{D}$, 
then, on any reflection hyperplane $\Pi$ of $u$ and $u'$,
\eqref{propertyofincident} holds.
If $\Sigma\neq \Sigma'$, then
\begin{equation}\label{thm1appl}\exists
\text{ a hyperplane $\Pi$ such that }|Bu^i|\leq \varepsilon\text{ on }\Pi\backslash \overline{B_R}.
\end{equation}
If we further assume that $\Sigma$ and $\Sigma'$ are polyhedral obstacles, then
\begin{multline}\label{thm2appl}\exists
\text{ $N$ hyperplanes $\Pi_1,\ldots,\Pi_N$ with linearly independent normals}\\\text{such that }
|Bu^i|\leq \varepsilon\text{ on }\Pi_i\backslash \overline{B_R}\text{ for any }i=1,\ldots,N.
\end{multline}
Hence proving the corresponding uniqueness results reduces to finding a suitable $u^i$ violating either \eqref{thm1appl} or \eqref{thm2appl}.

This situation is typical of inverse scattering problems. In this context $u^i$ is called the \emph{incident field} or \emph{wave}, whereas $u^s$ is called the \emph{scattered field} or \emph{wave}. A suitable condition at infinity, the \emph{radiation condition}, has to be imposed on $u^s$ and it usually guarantees the
validity of \eqref{conseqofradiation}. 
The scattered field may be measured either on the set $\tilde{D}$ (\emph{near-field data}) or at infinity (\emph{far-field data}). In fact, typically some kind of Rellich lemma holds, that is, if the far-field data of $u^s$ and $(u')^s$ coincide then $u^s$ and $(u')^s$ coincide on the whole $\tilde{G}$.

In the next, we just collect a few examples to show the applicability of the theory. The list is clearly not exhaustive.

\subsubsection{Acoustic waves}

Let $\omega\in \R$, $\omega> 0$, and
$$Au=\Delta u+\omega^2 u.$$
If $u$ solves $Au=0$, then $u$ is real analytic. We use two boundary conditions:
$Bu=B_Du=u$ for the Dirichlet case, with $\tilde{T}_{\Pi}(u)=-u$, or $Bu=B_Nu=\nabla u\cdot\nu$ for the Neumann case, with $\tilde{T}_{\Pi}(u)=u$, in both cases independently of $\nu$. We note that $B_D$ and $B_N$ satisfy (A.3)(c$_1$) as well.

Note that by setting $L=L_1=M$ and $A_ku=Au^k$ and $B_ku=Bu^k$ we can actually consider $M$ different measurements, that is, we can consider $M$ different experiments by using $M$ different incident fields $u^i_k$.

The radiation condition is the \emph{Sommerfeld radiation condition}. The far-field data corresponds to the so-called \emph{far-field pattern} of the scattered field. The Rellich lemma states that if the far-field pattern is zero, then the scattered wave is identically zero. The Sommerfeld radiation condition indeed implies the validity of
\eqref{conseqofradiation}.

For any \emph{incident direction} $d\in \mathbb{S}^{N-1}$, let the incident wave be given by a planar plane wave
$$u^i_d(x)=\rme^{\rmi\omega x\cdot d},\quad x\in\R^N.$$

Since $|u^i_d(x)|=1$ for any $x\in\R^N$, choosing $0<\varepsilon<1$, \eqref{thm1appl} can not hold and we have uniqueness with a single measurement (that is, with $M=1$ corresponding to using a single incident field) in the Dirichlet case.
In the Neumann case, $|B_Nu^i_d|=\omega|\nu\cdot d|$. If we obtain $N$ measurements (that is, $M=N$), corresponding to $N$ different incident fields $u^i_{d_k}$,
$k=1,\ldots,N$,  with $d_k$  linearly independent, then 
no matter what $\nu$ is, $|B_Nu^i_{d_k}|=\omega|\nu\cdot d_k|>0$ for some $k$ and, choosing $0<\varepsilon<\omega|\nu\cdot d_k|$, \eqref{thm1appl} can not hold. Thus we have uniqueness with $N$ measurements in the Neumann case, or in the mixed Dirichlet and Neumann case.
For polyhedral obstacles, no matter what $d$ is, 
$|B_Nu^i_{d}|=\omega|\nu_i\cdot d|>0$ for some $i$, where $\nu_i$ are the linearly independent normals of reflection hyperplanes. 
Choosing $0<\varepsilon<\omega|\nu_i\cdot d|$, \eqref{thm2appl} can not hold. Thus we have uniqueness with $1$ measurement in the Neumann case,
or in the mixed Dirichlet and Neumann case, for polyhedral obstacles.

\subsubsection{Electromagnetic waves}
Let $N=3$.

We let $u=(E,H)=(E_1,E_2,E_3,H_1,H_2,H_3)\in \C^6$.
Let $\omega\in \R$, $\omega> 0$, and
$$Au=(\nabla\times E-i\omega H,\nabla\times H+i\omega E).$$
If $Au=0$, then $\Delta E+\omega^2E=0$ and $\Delta H+\omega^2H=0$, hence any component of $u$ is real analytic.
We consider two boundary conditions: $Bu=B_E=E\times \nu$, for a \emph{perfectly electric conducting} scatterer, and 
$Bu=B_Hu=H\times \nu$, for a \emph{perfectly magnetic conducting} scatterer.
Then, independently of $\nu$, $\tilde{T}_{\Pi}(u)=\pm(-T_{\Pi}(E),T_{\Pi}(H))$,
with $+$ for $B_E$ and $-$ for $B_H$. We note that $B_E$ and $B_H$ satisfy (A.3)(c$_1$) as well.

Note that by setting $L=L_1=M$ and $A_ku=Au^k$ and $B_ku=Bu^k$ we can actually consider $M$ different measurements, that is, we can consider $M$ different experiments by using $M$ different incident fields $u^i_k$.

The radiation condition is the \emph{Silver-M\"uller radiation condition} which is equivalent to the validity of the
Sommerfeld radiation condition for any component of $E$ and $H$. Thus the Silver-M\"uller radiation condition implies the validity of
\eqref{conseqofradiation}. About near-field data, we can just measure $E$ or $H$ on $\tilde{D}$. About far-field data, we can just measure the far-field pattern of all components of $E$ or of all components of $H$.

About incident fields, let $u^i=(E^i,H^i)$ be given by
\begin{equation}\label{eq:planewave}
E^i(x)=\frac{\rmi}{\omega}\nabla\times\left(\nabla\times p\rme^{\rmi \omega x\cdot d}\right),\quad H^i(x)=\nabla\times p\rme^{\rmi \omega x\cdot d},\quad x\in\mathbb{R}^3,
\end{equation}
that is,
$$E^i(x)=-\omega \rmi(q\times d)\rme^{\rmi \omega x\cdot d} ,\quad H^i(x)=-\omega \rmi(p\times d) \rme^{\rmi \omega x\cdot d},\quad x\in\mathbb{R}^3.
$$
Here $u^i=u^i_{d,p}$
is the \emph{normalised electromagnetic plane wave} with \emph{incident direction} $d\in\mathbb{S}^2$ and \emph{polarisation vector} $p\in\mathbb{R}^3$, $p\neq 0$ orthogonal to $d$,
and $q=p\times d$. Then
$|Bu^i_{d,p}|=\omega|\nu\times (q\times d)|$. Arguing as for the Neumann acoustic case, we have uniqueness with two measurements (that is, $M=2$)
corresponding to $2$ different incident fields
$u^i_{d_1,p_1}$ and $u^i_{d_2,p_2}$ with $q_1\times d_1$ and $q_2\times d_2$ linearly independent for $B_E$ and
$p_1\times d_1$ and $p_2\times d_2$ linearly independent for $B_H$.
If $q_1\times d_1$ and $q_2\times d_2$ are linearly independent and $p_1\times d_1$ and $p_2\times d_2$ are also linearly independent (for 
example if $d_1=d_2$ and $p_1$ and $p_2$ are linearly independent)
then we have uniqueness with these two measurements also in the 
mixed electric and magnetic conducting cases.
For polyhedral obstacles, instead, we have uniqueness with just $1$ measurement in the electric conducting, magnetic conducting and
mixed electric and magnetic conducting cases.

\subsubsection{Elastic waves} Let $N\geq 2$. For a $\C^N$-valued function $u$, let 
\begin{equation}\label{Navier}
Au=\mu\Delta u+(\lambda+\mu)\nabla(\mathrm{div} (u))+\rho \omega^2 u.
\end{equation}
The equation $Au=0$ is the \emph{Navier equation}; here $\lambda$ and $\mu$ are the Lam\'e constants such that $\mu>0$ and $\lambda+2\mu>0$, $\rho>0$ is the density and 
$\omega>0$ is the frequency.

The symmetric gradient of $u$, $Eu=\frac{1}{2}(\nabla u+(\nabla u)^T)$, is the \emph{strain tensor} while the \emph{stress} $\sigma(u)$ is
$$\sigma(u)=2\mu Eu+\lambda \mathrm{tr}(Eu)I_N=
2\mu Eu+\lambda\mathrm{div}(u)I_N,$$
where $\mathrm{tr}$ denotes the trace and $I_N$ is the identity matrix.
The Navier equation $Au=0$ can be rewritten as
$$\mathrm{div}(\sigma(u))+\rho \omega^2 u=0,$$
where the $\mathrm{div}$ applies row by row.

By Helmholtz decomposition, any $u$ such that $Au=0$ satisfies
$$u=u_s+u_p$$
where $u_p$ is the \emph{longitudinal wave} $u_p$ and $u_s$ is the \emph{transversal wave}. We have that both $u_p$ and $u_s$ solve $Au=0$. Moreover,
\begin{equation}\label{longwave}
u_p=-\frac{\nabla\mathrm{div}(u)}{\omega_p^2}, \quad\Delta u_p+\omega_p^2 u_p=0,\quad  \omega_p^2=\frac{\rho\omega^2}{\lambda+2\mu}.
\end{equation}
Finally,
\begin{equation}\label{transwave}
u_s=\frac{\nabla\mathrm{div}(u)-\Delta u}{\omega_s^2},\quad
\Delta u_s+\omega_s^2u_s=0,\quad  \omega_s^2=\frac{\rho\omega^2}{\mu}.
\end{equation}
We immediately conclude that $u$ is real-analytic.

About boundary conditions, first of all we define the \emph{surface traction} $Tr(u)=\sigma(u)\nu$. For any $\nu$ and any
vector $V$, we denote $V_{\tau}=V-(V\cdot\nu)\nu$ the tangential component of $V$ with respect to $\nu$.
We consider the so-called \emph{third} and \emph{fourth boundary conditions}. Namely,
$$Bu=B^3u=(u\cdot \nu ,(Tr(u))_{\tau})$$
and
$$Bu=B^4u=(u_{\tau},Tr(u)\cdot \nu)$$
Then, in both cases independently of $\nu$, $\tilde{T}_{\Pi}(u)=\pm T_{\Pi}(u)$, with $+$ for the third boundary condition and $-$ for the fourth one.
We note that $B^3$ and $B^4$ satisfy (A.3)(c$_1$) as well.

Note that by setting $L=L_1=M$ and $A_ku=Au^k$ and $B_ku=Bu^k$ we can actually consider $M$ different measurements, that is, we can consider $M$ different experiments by using $M$ different incident fields $u^i_k$.

The radiation condition is the \emph{Kupradze radiation condition} which corresponds to both $u_s$ and $u_p$ solving the Sommerfeld radiation condition.
Thus the Kupradze radiation condition implies the validity of
\eqref{conseqofradiation}. About near-field data, we measure $u$ on $\tilde{D}$. About far-field data, we measure the far-field pattern of $u$, which is equivalent
to measuring the far-field pattern of both the longitudinal and transversal wave of $u$.

About incident fields, we let $u^i_{p,d}$ be a \emph{longitudinal plane wave} 
\begin{equation}\label{planarlong-INTRO}
u^i_{p,d}(x)=d\;\rme^{\rmi \omega_pd\cdot x},\qquad x\in \mathbb{R}^N,
\end{equation}
where $d\in \mathbb{S}^{N-1}$ is the \emph{incident direction}, and let $u^i_{s,d,q}$ be \emph{transversal plane wave}
\begin{equation}\label{planartrans-INTRO}
u^i_{s,d,q}(x)=q\; \rme^{\rmi \omega_sd\cdot x},\qquad x\in \mathbb{R}^N,
\end{equation}
where $q\in \mathbb{C}^N\backslash \{0\}$ is a unitary vector orthogonal to $d$.
Then as incident wave $u^i$ we can choose a linear combination of longitudinal and transversal plane waves, namely
\begin{equation}\label{planarcomb-INTRO}
u^i=c_pu^i_{p,d}+c_su^i_{s,d,q}
\end{equation}
for some $c_p,c_s\in\mathbb{C}$ such that $|c_p|^2+|c_s|^2=1$. In this case
$\|u^i(x)\|=1$ for any $x\in \mathbb{R}^N$.
If $c_p=0$ we call $u^i$ a pure transversal incident wave and if $c_s=0$ we call $u^i$ a pure longitudinal incident wave.

Then
$$u^i\cdot \nu=c_p\rme^{\rmi \omega_pd\cdot x}d\cdot\nu +
c_s\rme^{\rmi \omega_sd\cdot x}q\cdot\nu$$
and
$$u^i_{\tau}= c_p\rme^{\rmi \omega_pd\cdot x}d_{\tau} +
c_s\rme^{\rmi \omega_sd\cdot x}q_{\tau}.
$$
If $d$ and $q$ are row vectors,
$$\sigma(u^i)=c_p\rmi \omega_p\rme^{\rmi \omega_pd\cdot x} \left(2\mu d^Td+\lambda I_N\right)
+c_s\rmi \omega_s\rme^{\rmi \omega_sd\cdot x}\mu\left(d^Tq+q^Td\right)
$$
hence
\begin{multline*}
Tr(u^i)=\sigma(u^i)\nu=\\
c_p\rmi \omega_p\rme^{\rmi \omega_pd\cdot x} \left(2\mu (d\cdot \nu)d^T+\lambda \nu\right)
+c_s\rmi \omega_s\rme^{\rmi \omega_sd\cdot x}\mu\left((q\cdot\nu) d^T+(d\cdot\nu)q^T\right).
\end{multline*}
Therefore
$$
Tr(u^i)\cdot\nu=c_p\rmi \omega_p\rme^{\rmi \omega_pd\cdot x} \left(2\mu (d\cdot \nu)^2+\lambda \right)
+c_s\rmi \omega_s\rme^{\rmi \omega_sd\cdot x}2\mu(q\cdot\nu)(d\cdot\nu)
$$
and
$$(Tr(u^i))_{\tau}=c_p\rmi \omega_p\rme^{\rmi \omega_pd\cdot x} 2\mu (d\cdot \nu)d_{\tau}
+c_s\rmi \omega_s\rme^{\rmi \omega_sd\cdot x}\mu\left((q\cdot\nu) d_{\tau}+(d\cdot\nu)q_{\tau}\right).
$$
In conclusion,
\begin{multline*}
B^3(u^i)=
\Big(c_p\rme^{\rmi \omega_pd\cdot x}d\cdot\nu +
c_s\rme^{\rmi \omega_sd\cdot x}q\cdot\nu,\\
c_p\rmi \omega_p\rme^{\rmi \omega_pd\cdot x} 2\mu (d\cdot \nu)d_{\tau}
+c_s\rmi \omega_s\rme^{\rmi \omega_sd\cdot x}\mu\left((q\cdot\nu) d_{\tau}+(d\cdot\nu)q_{\tau}\right)
\Big)
\end{multline*}
and
\begin{multline*}
B^4(u^i)=\Big(c_p\rme^{\rmi \omega_pd\cdot x}d_{\tau} +
c_s\rme^{\rmi \omega_sd\cdot x}q_{\tau},\\
c_p\rmi \omega_p\rme^{\rmi \omega_pd\cdot x} \left(2\mu (d\cdot \nu)^2+\lambda \right)
+c_s\rmi \omega_s\rme^{\rmi \omega_sd\cdot x}2\mu(q\cdot\nu)(d\cdot\nu)
\Big).
\end{multline*}

We conclude that, for $B^4$ it is enough to use $1$ measurement corresponding to any pure longitudinal incident wave with any $d$ (that is, by choosing $c_p=1$ and $c_s=0$, for example). In fact, if $d_{\tau}\neq 0$, then the norm of the first component of $B^4(u^i)$ is $\|d_{\tau}\|> 0$ and we obtain a contradiction. Otherwise,
$|d\cdot\nu|=1$ and the modulus of the second component of $B^4(u^i)$ is $\omega_p(2\mu+\lambda)>0$ and we obtain a contradiction.

Instead, for $B^3$ we need to use two measurements to avoid the unfortunate case that both $d$ and $q$ are orthogonal to $\nu$. For example, one can use two pure transversal incident waves $u^i_{s,d_1,q_1}$ and $u^i_{s,d_2,q_2}$ with $d_1=d_2$ and $q_1$ and $q_2$ linearly independent. We note that these measurements
give uniqueness even in the mixed
third and fourth boundary conditions.

Finally, in the case of polyhedral obstacles,
one measurement is enough, by choosing as incident wave any pure transversal incident wave or any pure longitudinal incident wave, even in the mixed
third and fourth boundary conditions.

\end{document}